\documentclass[12pt,reqno]{article}

\usepackage[usenames]{color}
\usepackage{amssymb}
\usepackage{mathtools}
\usepackage{graphicx}
\usepackage{epsfig}
\usepackage[colorlinks=true,
linkcolor=webgreen,
filecolor=webbrown,
citecolor=webgreen]{hyperref}

\definecolor{webgreen}{rgb}{0,.5,0}
\definecolor{webbrown}{rgb}{.6,0,0}

\usepackage{amsmath}
\usepackage{amsthm}
\usepackage{amsfonts}

\DeclareMathOperator{\erfi}{Erfi}

\newcommand{\seqnum}[1]{\href{http://oeis.org/#1}{\underline{#1}}}

\begin{document}

\theoremstyle{plain}
\newtheorem{theorem}{Theorem}
\newtheorem{corollary}[theorem]{Corollary}
\newtheorem{lemma}[theorem]{Lemma}
\newtheorem{proposition}[theorem]{Proposition}
\newtheorem{result}[theorem]{Result}

\theoremstyle{definition}
\newtheorem{definition}[theorem]{Definition}
\newtheorem{example}[theorem]{Example}
\newtheorem{conjecture}[theorem]{Conjecture}

\theoremstyle{remark}
\newtheorem{remark}[theorem]{Remark}

\begin{center}
\vskip 1cm{\LARGE\bf 
Generating Functions for Domino Matchings in the $2\times k$ Game of Memory\\
}
\vskip 1cm 
\large
Donovan Young\\
St Albans, Hertfordshire \\
AL1 4SZ \\
United Kingdom\\
\href{mailto:donovan.m.young@gmail.com}{\tt donovan.m.young@gmail.com} \\
\end{center}

\vskip .2 in

\begin{abstract}
When all the elements of the multiset $\{1,1,2,2,3,3,\ldots,k,k\}$ are
placed in the cells of a $2\times k$ rectangular array, in how many
configurations are exactly $v$ of the pairs directly over top one
another, and exactly $h$ directly beside one another --- thus forming
$1\times 2$ or $2\times 1$ dominoes? We consider the sum of matching numbers over the
graphs obtained by deleting $h$ horizontal and $v$ vertical vertex
pairs from the $2\times k$ grid graph in all possible ways, providing
a generating function for these aggregate matching polynomials. We use
this result to derive a formal generating function enumerating the
domino matchings, making connections with linear chord diagrams.
\end{abstract}

\section{Introduction}

The game of memory consists of the placement of a set of distinct
pairs of cards in a rectangular array. The present author \cite{DY}
considered the enumeration of the configurations in which exactly $p$
of the pairs are placed directly beside, or over top of one another,
thus forming $1\times 2$ or $2\times 1$ dominoes. In this paper we
consider the case of $2\times k$ arrays in more detail. In Figure
\ref{fourways} we show a configuration of the case $k=4$ with $h=1$
horizontal dominoes, and $v=1$ vertical dominoes.
\begin{figure}[h]
\begin{center}
\includegraphics[bb=0 0 516 92, height=0.9in]{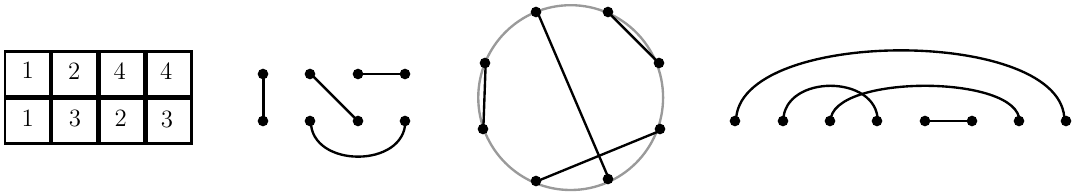}
\end{center}
\caption{A configuration for the $2\times 4$ array with one horizontal
  and one vertical domino is shown in four different
  representations. From left to right: a placement of paired cards in
  a game of memory, a Brauer diagram whose links correspond to the
  pairs, a chord diagram, and finally as a linear chord diagram
  resulting from breaking the circle in the chord diagram at its
  Westernmost point.}
\label{fourways}
\end{figure}
The enumeration of these configurations always carries a factor of
$k!$, which counts the orderings of the $k$ distinguishable pairs. It
is therefore easier to drop this factor, and thus treat the pairs as
indistinguishable. We can then think of the domino enumeration problem
in different ways: as a Brauer diagram\footnote{Terada \cite{IT} and
  Marsh and Martin \cite{MM} have considered Brauer diagrams in the
  context of combinatorics.}, as a chord diagram (cf.\ Krasko and
Omelchenko \cite{KO}), or unfolded as a linear chord diagram
(cf.\ Cameron and Killpatrick \cite{CK}). In the present paper we
provide a generating function for the numbers $D_{k,v,h}$ which count,
considering the pairs to be indistinguishable, the number of
configurations with $h$ horizontal, and $v$ vertical dominoes. It is
clear that $D_{k,v,h}=0$ for $h+v>k$; considering the $D_{k,v,h}$ as
matrices with $v\geq 0$ indexing rows, and $h\geq 0$ indexing columns,
the entries below the anti-diagonal are therefore all zero. The first few values
are as follows:
\begin{equation}\nonumber
  D_{0,v,h} = 1,\quad
  D_{1,v,h} = \begin{pmatrix}
    0& 0\\
    1& 
  \end{pmatrix},\quad
  D_{2,v,h} = \begin{pmatrix}
    1& 0& 1\\
    0& 0\\
    1&
  \end{pmatrix},\quad
   D_{3,v,h} = \begin{pmatrix}
    2& 4& 2& 0\\
    4& 0& 2&\\
    0& 0\\
    1&
  \end{pmatrix},
\end{equation}
where we have omitted the aforementioned zero entries. The sum of the
numbers on the anti-diagonal are the Fibonacci numbers, which count
the domino tilings of the $2\times k$ array\footnote{Graham, Knuth,
  and Patashnik \cite[p. 320]{GKP} give an account.}. The sum of all
numbers in the matrix\footnote{We recall that the double factorial is
  given by $n!! \coloneqq \prod_{k=0}^{\lceil{n\over2}\rceil-1} (n-2k)$, and
  we define $0!!=(-1)!!=1$.} is $(2k-1)!!$, which is the number of
ways of placing the cards, modulo re-labelling of the pairs.

The strategy we will employ is to consider the matching numbers of the
$2\times k$ grid graph, whose vertices represent the $2\times k$ array
of cards, and whose edges define the possible domino matchings. In
Figure~\ref{explaingrid} we show this grid graph for the case
$k=4$. The present author \cite[Section 3, Theorem 4]{DY} provided a
method for computing the number of $0$-domino configurations (i.e.,
configurations without any matched pairs) on any analogous graph
$G$. Let $\rho_j$ be the number of $j$-edge matchings\footnote{A
  $j$-edge matching is defined as a set of $j$ pairwise non-adjacent
  edges, none of which are loops.} on $G$. Then the number of
$0$-domino configurations is given by
\begin{equation}\nonumber
\sum_{j=0}^n (-1)^j(2n-2j-1)!! \,\rho_j,
\end{equation}
where $n$ is the number of pairs. We may therefore compute the
$D_{k,v,h}$ by computing the matching numbers for the graphs which
arise from removing $v$ vertical, and $h$ horizontal vertex pairs (and
their incident edges), from the $2\times k$ grid graph in all possible
ways.

\section{Preliminaries}

The {\it board} associated with the grid graph is defined as follows
(cf.\ Riordan \cite[p.\ 163]{JR1}). In Figure~\ref{explaingrid}, we
show the board for the case $k=4$. We color the vertices of the grid
graph black and white, in a checkerboard pattern. The columns (rows)
of the board represent the black (white) vertices. The cells of the
board represent the edges of the graph. The vertical edges correspond
to the cells on the central diagonal, while the horizontal edges
correspond to the upper and lower diagonals. The rook or matching
polynomial $g(x) = \sum g_j x^j$ encodes the number $g_j$ of $j$-edge
matchings on the graph and enjoys two important properties. The first
comes from partioning the set of all matchings into two sets: those
which contain a specific edge, and those which do not. One can then
{\it develop} the associated board using the property that the rook
polynomial of a board $B$ is equal to that of $B$ with a given cell
removed (corresponding to the absence of the specific edge in the
matching), plus $x$ times the rook polynomial of $B$ with the entire
row and column containing that cell removed (corresponding to the
presence of the specific edge in the matching). An example is shown in
Figure~\ref{develop}. The second property stems from graphs consisting
of disconnected components; the disconnectedness implies that the
number of matchings of each component are independent of one
another. We then have that if a board can be separated into regions
whose cells share no common row or column with another region (as in
the right hand side of Figure~\ref{develop}), the rook polynomial
factorizes into a product of the polynomials for the regions.

\begin{figure}[t]
\begin{center}
\includegraphics[bb=0 0 257 93, height=1.25in]{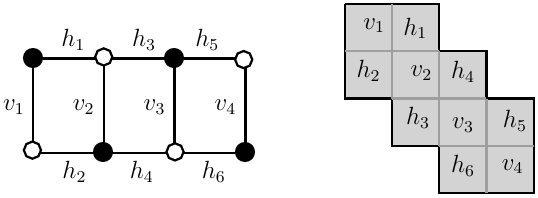}
\caption{The $2\times 4$ grid graph is shown on the left, while the
  corresponding board is shown on the right. The mapping of the edges
  of the graph to the cells of the board is also displayed.}
\label{explaingrid}
\end{center}
\end{figure}

\begin{figure}[t]
  \begin{center}
    \includegraphics[bb=0 0 288 71, height=1.in]{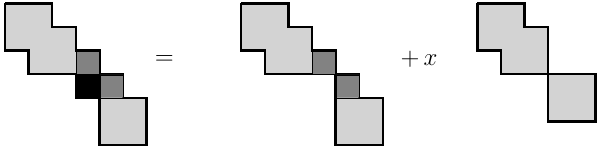}
\caption{The board on the left is developed using the black cell. The
  cells in the row and column containing the black cell are shown in
  dark grey. The boards on the right hand side both factorize into the
  product of two rook polynomials.}
\label{develop}
\end{center}
\end{figure}
\begin{figure}[h]
\begin{center}
\includegraphics[bb=0 0 293 112, height=1.5in]{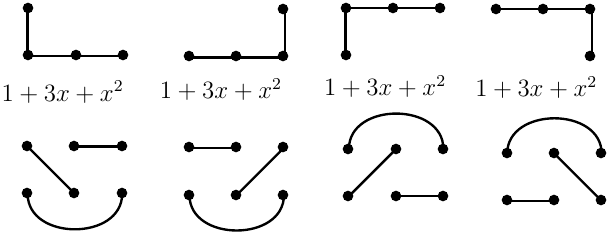}
\caption{On the bottom row, the four configurations counted by
  $D_{3,0,1}$ are shown as Brauer diagrams. On the top row, the set of
  graphs ${\cal G}_{3,0,1}$ corresponding to the removal of the vertex
  pair corresponding to the domino, along with their matching
  polynomials, are shown. The sum of these polynomials yields ${\cal
    T}_{3,0,1}(x) = 4(1+3x+x^2)$.}
\label{exampleTD}
\end{center}
\end{figure}

Riordan \cite[p.\ 230]{JR,JR1} (McQuistan and Lichtman \cite{ML} give
connections to dimer models in Physics) provided the generating
function for the rook polynomials of the $2\times k$ grids
\begin{equation}\label{riordan}\nonumber
T(x,y) \coloneqq \frac{1-xy}{1-y-2xy-xy^2+x^3y^3} = \sum_{k=0}^\infty T_k(x) y^k ,
\end{equation}
where $T_k(x)$ are the rook polynomials. Riordan also provided similar
generating functions for several related grids, whose boards are shown
in Figure~\ref{shapes},
\begin{align}\nonumber
  &s(x,y) \coloneqq \frac{T(x,y)}{(1-xy)^2},\quad r(x,y) \coloneqq (1-xy)\, s(x,y),
  \quad R(x,y) \coloneqq y\,r(x,y),\\\nonumber
  &S(x,y) \coloneqq \left(1-2xy-xy^2+x^3y^3\right)s(x,y).
\end{align}
For example, the rook polynomial corresponding to the board on the
left hand side of Figure~\ref{develop} is $R_3(x)\, r_3(x) +
x\,T_3(x)\, T_2(x)$.

\section{Generating functions for matching numbers}

We now turn our attention to the set ${\cal G}_{k,v,h}$ of graphs
which arise when $h$ horizontal, and $v$ vertical vertex pairs are
removed, in all possible ways, from the $2\times k$ grid graph. The
four graphs belonging to ${\cal G}_{3,0,1}$ are displayed in
Figure~\ref{exampleTD}. Each of the graphs $g \in {\cal G}_{k,v,h}$
will have an associated matching polynomial, which with a slight abuse
of notation, we will denote $g(x)$. The combinatorial object of
interest will be the sum of these polynomials
\begin{equation}\nonumber
  {\cal T}_{k,v,h}(x) \coloneqq \sum_{g\in {\cal G}_{k,v,h}} g(x).
\end{equation}
\begin{figure}[t]
\begin{center}
\includegraphics[bb=0 0 481 174, height=2.in]{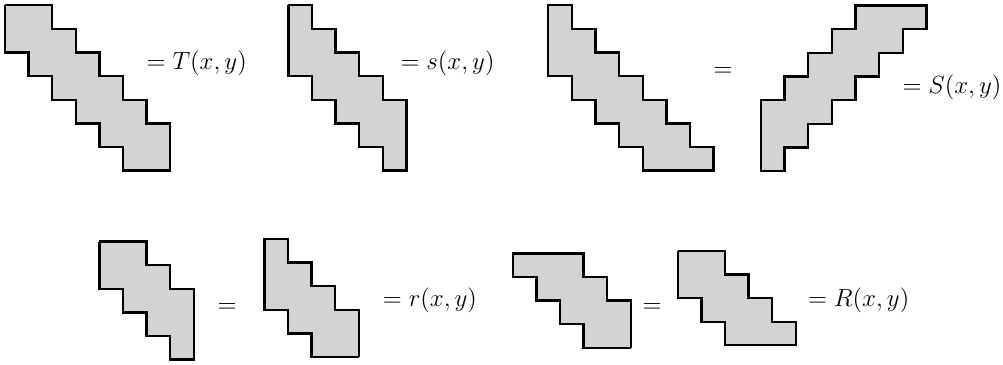}
\end{center}
\caption{The boards which arise when the calculus of the rook
  polynomial is applied to the boards shown in Figure~\ref{gridboard}.}
\label{shapes}
\end{figure}
\begin{theorem}
The generating function for the ${\cal T}_{k,v,h}(x)$ is given
by\footnote{\label{ftp}The power of $y$ is taken to be $k-v-h$; this is a useful
  parameterization for computing the generating function for the
  $D_{k,v,h}$.}
\begin{align}\nonumber\label{ThmFirst}
&{\cal T}(x,y,w,z)\coloneqq \sum_{k,h,v\geq 0} {\cal T}_{k,v,h}(x) \,y^{k-v-h} w^v z^h\\
  &=\frac{1 - xy - z}{1-(1+2 x) y-z - w (1 - xy - z)+(x y+z)
    (x^2 y^2 - (1 - z) z - y (1 - 2 x z))},
\end{align}
where the power of the variable $w$ (respectively $z$) corresponds to
the number of removed vertical (respectively horizontal) vertex pairs,
while the power of the variable $y$ corresponds to the number of
remaining vertex pairs after these removals have taken place. The
power $j$ of the variable $x$ corresponds to $j$-edge matchings in the
resulting graphs.
\end{theorem}

\begin{proof}
   The removal of a horizontal vertex pair from the grid graph
   corresponds to the deletion of a cell on the lower or upper
   diagonal of the board, together with its entire row and column. In
   Figure~\ref{gridboard}, boards resulting from the deletion of two
   horizontal vertex pairs are shown.  When the boards are developed
   using the black cells in Figure~\ref{gridboard}, various shapes
   arise; these are shown in Figure~\ref{shapes}. For example, the
   configuration shown on the left in Figure~\ref{gridboard}
   corresponds to
   \begin{align}\nonumber\label{example}
     &R(x,y)\cdot S(x,y) \cdot r(x,y) + xy\, T(x,y)\cdot R(x,y) \cdot r(x,y)\\
     &+ R(x,y) \cdot r(x,y)\cdot xy\, T(x,y)\nonumber
     + x^2y^2\, T(x,y) \cdot T(x,y) \cdot T(x,y)\\
     &= yr^2S + 2xy^2 Tr^2 + x^2y^2 T^3,
   \end{align}   
   where the $\cdot$ denotes ordinary multiplication, and has been
   included to aid in the following explanation. The first term
   corresponds to the removal of both the black cells, the second
   (third) term to the additional removal of the row and column
   containing the first (second) black cell, and the last term to the
   removal of the rows and columns containing both black cells. The
   multiplication also accounts\footnote{The reader is referred to
     Flajolet and Sedgewick \cite[p.\ 16]{FS} for an account of the
     basic symbolic method.}  for the ordered sum over all possible
   positions of the removed horizontal vertex
   pairs\footnote{\label{fn}The exception is when the two removed
     horizontal vertex pairs are directly over top of one another in
     the grid graph; these coincident configurations will be accounted
     for below.}. When the row and column containing a black cell is
   removed, the overall board is shortened, and hence earns a factor
   of $y$; this is why the expression $xy$ appears for each such
   removal. It is straightforward to see that the corresponding
   expression for the board shown on the right in
   Figure~\ref{gridboard} differs from the above only in the first
   term, which becomes $R(x,y) \cdot s(x,y) \cdot R(x,y) = y^2r^2s$.
\begin{figure}[t]
\begin{center}
\includegraphics[bb=0 0 217 196, height=1.5in]{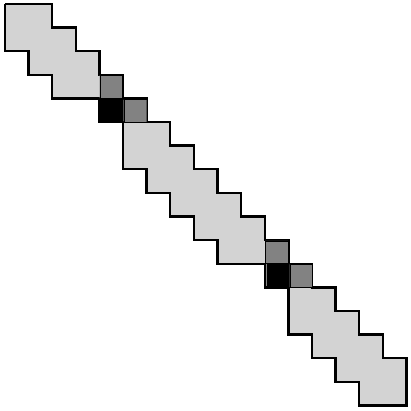}
\includegraphics[bb=0 0 217 196, height=1.5in]{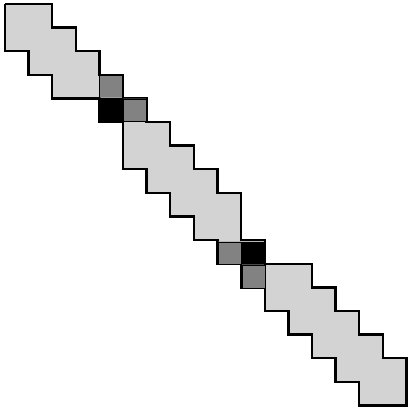}
\end{center}
\caption{Boards corresponding to removing two horizontal vertex pairs
  from the grid graph. On the left both removed pairs correspond to
  cells on the upper diagonal. On the right, one upper, and one lower
  diagonal pair have been removed. The black cells are used in further
  developing the boards using the calculus of the rook polynomial, in
  the same way as shown in the example in Figure~\ref{develop}.}
\label{gridboard}
\end{figure}

We now focus on generalising to the removal of an arbitrary number
$\tilde h$ of non-coincident (see Footnote~\ref{fn}) horizontal vertex
pairs from the grid graph. In keeping with previous notation, we call
this set of graphs $\widetilde{\cal G}_{k,0,\tilde h}$. We begin by accounting
for the terms which arise from the first step in the development of
the associated boards, i.e. when the black cells are removed from the
boards. For the time being, we exclude the terms which arise when any
rows and columns containing those black cells are also removed. We
introduce the variable $z$, the power of which corresponds to the
number of horizontal vertex pairs removed from the grid graph. When
$\tilde h=0$, we have simply $T$. When $\tilde h=1$, we have $2zyr^2$,
where the factor of 2 arises as the removed horizontal vertex pair can
correspond to a cell on the upper, or the lower diagonal of the
board. As we have seen in the previous paragraph, for $\tilde h=2$, we
have $2z^2(yr^2S+y^2r^2s)=2z^2yr^2(S+ys)$, where again the factor of 2
accounts for swapping the diagonals which the two cells (corresponding
to the removed vertex pairs) are taken from. It follows that the
general term for $\tilde h >0$ is $2yzr^2 \left(z(S+ys)\right)^{\tilde
  h-1}$. Performing a sum over $\tilde h \geq 0$ we obtain the
expression
   \begin{equation}\label{jjghs}
     T + \frac{2yzr^2}{1-z(S+ys)} .
   \end{equation}

We now account for those terms which arise in the development of the
boards whenever the row and column containing a black cell are further
removed. As was explained above, the calculus of the rook polynomial
implies that we gain a factor of $xy$ for each such a removal. In
fact, we gain a factor of $2xyz$; the 2 since the black cell could be
on the upper, or lower diagonal, and the $z$ to account for the
corresponding removal of the vertex pair from the grid graph. Once one
row and column have been removed, the board is split into two
independent boards. If we remove the columns and rows of $q-1$
different black cells, we will have $q$ independent boards; the
corresponding terms are then given by
   \begin{equation}\nonumber
     (2xyz)^{q-1}\left(T + \frac{2yzr^2}{1-z(S+ys)} \right)^q.
   \end{equation}
We can also interpret this expression from the point of view of the
grid graph itself. The factor $(2xyz)^{q-1}$ corresponds to the removal
of $q-1$ horizontal vertex pairs, where the coincident edge below or
above is chosen in the matching. The factor 2 comes from choosing
either the upper or the lower vertex pair to remove. This leaves us to
deal with $q$ separate and independent grids. Each of these $q$ grids
is enumerated by a factor of Equation~(\ref{jjghs}), where $T$
corresponds to the case where no further horizontal vertex pairs are
removed. The fraction, on the other hand, corresponds to the case
where at least one further horizontal vertex pair is removed (but the
coincident edge is not chosen in a matching).
   
   Taking the sum over $q>0$, we obtain
   \begin{equation}\nonumber
     {\cal X} \coloneqq\sum_{k,\tilde h\geq 0}\,y^{k-\tilde h} z^{\tilde h}\
     \left(\sum_{g\in\widetilde {\cal G}_{k,0,\tilde h}} g(x)\right)
     =\frac{T + \frac{2yzr^2}{1-z(S+ys)}}{1-2xyz\left(T +
  \frac{2yzr^2}{1-z(S+ys)}\right)}.
   \end{equation}
   
   It remains to account for the removal of vertical vertex pairs, and
   coincident horizontal vertex pairs, from the grid graph. The removal
   of a single pair of coincident horizontal vertex pairs is
   equivalent to the removal of two neighboring vertical vertex
   pairs. The effect of either of these removals on the board is again to
   break it into two independent boards. Thus we find that the removal of
   coincident horizontal and vertical vertex pairs is accounted for as
   follows:
   \begin{equation}\nonumber
     \sum_{j=0}^\infty (w + z^2)^j {\cal X}^{j+1}
     =  \frac{T + \frac{2yzr^2}{1-z(S+ys)}}{1-(2xyz+z^2+w)
       \left(T + \frac{2yzr^2}{1-z(S+ys)}\right)},
   \end{equation}
   where the coefficient of $w^v$ corresponds to the removal of $v$
   vertical vertex pairs. Simplifying this expression, we obtain
   Equation~(\ref{ThmFirst}).   
\end{proof}

\section{From matching numbers to domino-counting generating functions}

We now use the result of Theorem~\ref{ThmFirst} to compute the number
of configurations with exactly $h$ horizontal, and $v$ vertical
dominoes. Let the $\rho_j(k,v,h)$ be defined\footnote{We use the notation $[y^n]\,f(y)$
  to represent the coefficient of $y^n$ in the Taylor expansion of
  $f(y)$.} as follows:
\begin{equation}\nonumber
\rho_j(k, v, h) \coloneqq [x^j] \mathcal{T}_{k, v, h}(x).
\end{equation}
We remind the reader the combinatorial significance of this
quantity. The set ${\cal G}_{k,v,h}$ of graphs is obtained from
removing $v$ vertical, and $h$ horizontal vertex pairs from the
$2\times k$ grid graph in all possible ways. Each graph $g\in {\cal
  G}_{k,v,h}$ has a number $g_j$ of $j$-edge matchings. We then have
that
\begin{equation}
  \rho_j(k, v, h) = \sum_{g\in{\cal G}_{k,v,h} }g_j.
\end{equation}
As mentioned
in the Introduction, the number $D_{k,h,v}$ of configurations with
exactly $h$ horizontal, and $v$ vertical dominoes is given by
\begin{equation}\label{MyOldThm}
D_{k,v,h}=\sum_{j=0}^n (-1)^j(2n-2j-1)!! \,\rho_j(k,v,h),
\end{equation}
where $n=k-h-v$. We define the corresponding generating function in
the usual way,
\begin{equation}\nonumber
D(y,w,z) \coloneqq \sum_{k,h,v\geq 0} D_{k,v,h}\,y^kw^vz^h.
\end{equation}
We now translate Equation~(\ref{MyOldThm}) into an operation on ${\cal T}(x,y,w,z)$.
\begin{theorem}
  The generating function $D(y,w,z)$ may be obtained using the
  following integral representation:
\begin{equation}\nonumber
 D(y,w,z) = \int_0^\infty dt\,e^{-t}
\frac{1}{2\pi i} \oint_{C_\epsilon} \frac{dx}{x\sqrt{1+2x}}\,
      {\cal T}\left(\frac{x}{t},\frac{-yt}{x},yw,yz\right),
\end{equation}
where the contour integral with respect to $x$ is taken around a small
circle containing the origin.
\end{theorem}
\begin{proof}
We consider the coefficient of $y^n$ in the Taylor expansion of ${\cal T}(x,y,w,z)$,
and define $\Omega_{j,n}(w,z)\coloneqq [x^jy^n]\, {\cal T}(x,y,w,z)$. We have
\begin{equation}\nonumber
[y^n]\, {\cal T}(x,y,w,z)=\sum_{j=0}^n \Omega_{j,n}(w,z)\, x^j =
\sum_{j=0}^n \left(\sum_{h,v\geq 0} \rho_j(n+h+v, v, h) \, w^v z^h \right)x^j
  .
\end{equation}
Under the integration in $t$, the replacements $x\to x t^{-1}$ and $y\to
yt$ dress this result by a factor of $\int_0^\infty dt\, t^{n-j} e^{-t}=(n-j)!$
\begin{equation}\label{thiseq}
  [y^n]\int_0^\infty dt \,e^{-t}\, {\cal T}(xt^{-1},yt,w,z) 
  =\sum_{j=0}^n
  (n-j)!\,\Omega_{j,n}(w,z)\, x^j.
\end{equation}
We now consider the coefficient of $x^n$ in the expression obtained by
multiplying Equation~(\ref{thiseq}) with $(1+2x)^{-1/2}$
\begin{align}\nonumber
  &[x^n] \left( \sum_{q=0}^\infty (-1)^q \frac{(2q-1)!!}{q!} \,x^q \right)
  \left( \sum_{j=0}^n (n-j)!\,\Omega_{j,n}(w,z)\, x^j \right) \\\nonumber
  &= (-1)^n \sum_{j=0}^n (-1)^j (2n-2j-1)!! \, \Omega_{j,n}(w,z).
\end{align}
We may, therefore, compute this quantity by further scaling $y \to
y/x$ and taking the residue at the origin after an overall
multiplication by $x^{-1}$. The factor of $(-1)^n$ is absorbed by a
final replacement $y \to -y$. The variables $w$ and $z$ are also
scaled by $y$, so that the unnatural parameterization discussed in
Footnote~\ref{ftp} is rectified.
\end{proof}

\begin{corollary}
The generating function $D(y,w,z)$  is given by
\begin{align}\nonumber
D(y,&w,z)= \\ \nonumber
 &\int_0^\infty dt\frac{e^{-t}}{(1+(1-w)y -(1-z)^2y^2)
    \sqrt{1-\frac{2 t y (1-(1-z)y)}{(1+(1- z)y) (1+(1-w)y -(1-z)^2y^2)}}}\\
\nonumber
  &=\sum_{j=0}^\infty (2j-1)!!
  \frac{ y^j (1-(1-z)y)^j}{(1+(1-z)y)^j(1+(1-w)y -(1-z)^2y^2)^{j+1}}.
\end{align}
\end{corollary}
\begin{proof}
  We note from Equation~(\ref{ThmFirst}) that ${\cal
    T}\left(\frac{x}{t},\frac{-yt}{x},yw,yz\right)= Ax/(Bx+Ct)$, where
  $A,B$, and $C$ are functions of $y,w$, and $z$. The contour
  integration replaces $x \to -Ct/B$ in the factor
  $(1+2x)^{-1/2}$. The integration over $t$ is interpreted as acting
  on the Taylor expansion of the resulting expression.
\end{proof}
This is not a convergent series\footnote{It may be interpreted as the
  real part (taking $y,w$, and $z\in \mathbb{R}$) of the expansion of
  the following expression, asymptotic in $y^{-1}$:
  $\sqrt{\frac{2(1+(1- z)y)}{y(1-(1- z)y)(1+(1-w) y-(1-z)^2y^2)}}
  F\left( \sqrt{\frac{(1+(1- z)y)(1+(1-w)
      y-(1-z)^2y^2)}{2y(1-(1-z)y)}}\right)$, where $F$ is Dawson's
  integral; Nijimbere \cite{VN} gives a modern account of the
  asymptotic expansion of this function and its relatives.}, and hence
we cannot benefit from an analytic generating function, with which
questions about asymptotic behavior could easily be answered. In order
to convert this formal generating function into a convergent series,
we could take an inverse Laplace transform in $y$ to form an
exponential generating function
\begin{equation}\nonumber
E(y,w,z) \coloneqq {\cal L}^{-1}\left\{ y^{-1} D(y^{-1},w,z) \right\}.
\end{equation}
Performing this transform is not straightforward, however in the
simple case of counting only vertical dominoes, it is feasible, and
yields a well-known result \seqnum{A055140},
\begin{equation}\label{byvonly}
D(y,w,1) = \sum_{j=0}^\infty \frac{(2j-1)!!\,y^j}{(1+(1-w)y)^{j+1}}
\to E(y,w,1)=
\frac{e^{y(w-1)}}{\sqrt{1-2y}},
\end{equation}
which counts the number of matchings of $2k$ people with partners (of
either sex) such that exactly $v$ couples are left
together. Unfortunately, we have been unable to perform the transform
for the case of counting only horizontal dominoes \seqnum{A325754}
\begin{equation}\nonumber
 D(y,1,z) =  \frac{1}{\left(1-(1-z)y\right)}
\sum_{j=0}^\infty   \frac{(2j-1)!! \,y^j}{\left(1+(1-z)y\right)^{2j+1}},
\end{equation}
but in the next section we derive the exponential generating function
by appealing to known results for the $1\times 2k$ problem.

It is also interesting to consider the case where vertical and
horizontal dominoes are not distinguished, i.e., $D(y,z,z)$. The
present author \cite[Section 4]{DY} considered this sequence previously. We can
now provide a generating function for these numbers \seqnum{A325753}
\begin{equation}\nonumber
  D(y,z,z) = 
  \sum_{j=0}^\infty
  \frac{(2j-1)!!\,y^j\,(1-(1-z)y)^j}
       {(1+(1-z)y)^j\left(1+(1-z)y-(1-z)^2y^2\right)^{j+1}}.
\end{equation}
The present author \cite[Section 4.2]{DY} also made several
conjectures\footnote{A. Howroyd \cite{AH} has proven several of
  these.} for generating functions for the so-called $(k-\ell)$-domino
configurations, when the number of dominoes is $\ell$ less than the
maximum value $k$. These can now be readily calculated as follows:
\begin{equation}\nonumber
  {\cal F}_{\ell}(y) \coloneqq \sum_{k\geq 0} \Biggl(\sum_{\substack{h,v\geq 0\\h+v=k-\ell}}
  D_{k,v,h}\Biggr)\, y^k
  =[z^{-\ell}] \lim_{z\to\infty} D\left(\frac{y}{z},z,z\right).
\end{equation}
The first few such generating functions are
\begin{align}\nonumber
  {\cal F}_0 = \frac{1}{1-y-y^2},\qquad {\cal F}_1 =
  \frac{2y^3}{(1-y)(1-y-y^2)^2},\qquad {\cal F}_2 =
  \frac{y^2(1+3y+6y^2+y^3+3y^4)}{(1-y)^2(1-y-y^2)^3}.
\end{align}
The function ${\cal F}_0$ is the generating function for the Fibonacci
numbers, giving the number of domino tilings. The function ${\cal
  F}_1$ is that for the path length of the Fibonacci tree of order
$k$, \seqnum{A178523}. The sequences corresponding to the cases
$\ell=2,\dots,5$ appear as \seqnum{A318267} to \seqnum{A318270},
respectively; before the results presented here were available these
were obtained by using data produced by a computer program to fix the
coefficients in the numerators of the generating functions, based on a
guess for the pattern of the denominators. Hence only a small range of
values of $\ell$ were achieved.

\section{Connections to linear chord diagrams}

Kreweras and Poupard \cite{KP} solved the problem of counting the
$h$-domino configurations on the $1\times 2k$ grid graph (i.e., a path
of length $2k$). Cameron and Killpatrick \cite{CK} recently revisited
this case in the context of linear chord diagrams, and provided a
derivation of the corresponding exponential generating function. Let
$L_{k,h}$ be the number of $h$-domino configurations on the $1\times
2k$ grid graph. We seek to establish a correspondence with the number
$D_{k,h}$ of $h$ horizontal domino configurations on the $2\times k$
grid graph, where we allow any number of vertical dominoes.
\begin{theorem}\label{LtoD}
The numbers $L_{k,h}$ and $D_{k,h}$ are related by the following recursion relation:
  \begin{equation}\nonumber
D_{k,h} = D_{k-1,h} + L_{k,h} - D_{k-1,h-1}.
    \end{equation}
\end{theorem}
\begin{proof}
We begin by unfolding the the vertices of the $2\times k$ grid graph,
as shown in Figure~\ref{unfold}, to give the vertices of the $1\times
2k$ grid graph. We then note that the central pair of vertices does
not correspond to a horizontal domino in the $2\times k$ graph, but
rather to a vertical one. The configurations counted by $L_{k,h}$
may be divided into two sets: those with a domino on the central pair
and those without. Those configurations with a domino on the central
pair are counted by $D_{k-1,h-1}$, as the central pair is effectively
deleted, leaving the $2\times (k-1)$ grid graph with $h-1$ horizontal
dominoes. In Figure~\ref{DfromL}, we provide a pictorial interpretation
of this relation. The configurations counted by $D_{k,h}$ can be
similarly divided, but since the central pair this time represents a
vertical domino, those configurations with this vertical domino are
equal in number to $D_{k-1,h}$.
\end{proof}
\begin{figure}[t]
\begin{center}
\includegraphics[bb=0 0 348 70, height=0.8in]{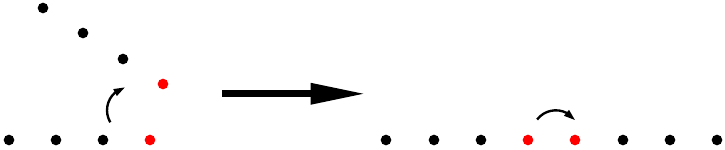}
\caption{The $2\times k$ grid graph is unfolded to produce the
  $1\times 2k$ grid graph. The vertices marked in red comprise a
  vertical domino, which becomes horizontal upon unfolding.}
\label{unfold}
\end{center}
\end{figure}

\begin{figure}[h]
\begin{center}
\includegraphics[bb=0 0 437 84, height=0.9in]{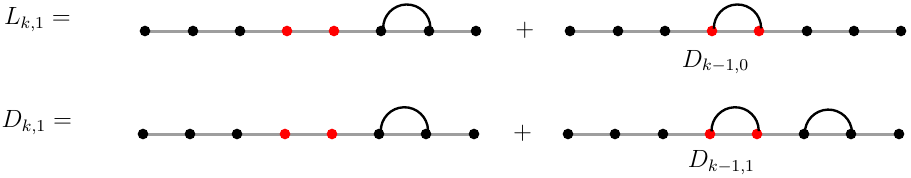}
\caption{The relationship between the numbers, $L_{k,h}$ and $D_{k,h}$,
  of $h$-horizontal-domino configurations on the $1\times 2k$, and
  $2\times k$ grid graphs, respectively, is shown for the case $h=1$.
}
\label{DfromL}
\end{center}
\end{figure}

Jovovic \cite{JV}, and Cameron and Killpatrick \cite{CK}, provide the
exponential generating function
\begin{equation}\nonumber
  L(y,z) \coloneqq
  \sum_{k,h\geq 0} L_{k,h} \frac{y^k}{k!}\, z^h
  = \frac{e^{(\sqrt{1-2 y}\,-\,1) (1-z)}}{\sqrt{1-2 y}}.
\end{equation}
Translating the recursion relation from Theorem~\ref{LtoD} into a
differential equation for the exponential generating function, we obtain
\begin{equation}\label{ODE}
  \frac{\partial E(y,1,z)}{\partial y} - (1-z) E(y,1,z) =
  \frac{\partial L(y,z)}{\partial y}.
\end{equation}
This elementary, non-homogeneous, first-order ODE may be solved using an
integrating factor. 
\begin{corollary}
  The exponential generating function for $D_{k,h}$ is as follows:
\begin{align}\label{byhonly}
E(y,1,z)& =  \frac{e^{(\sqrt{1-2 y}\,-\,1) (1-z)}}{\sqrt{1-2 y}}\\\nonumber
  &-e^{(y-2) (1-z)} \sqrt{\frac{\pi}{2}} \sqrt{1-z}
  \left(
\erfi
\left( \frac{(\sqrt{1-2 y}+1) \sqrt{1-z}}{\sqrt{2}} \right)
-\erfi (\sqrt{2} \sqrt{1-z})
  \right),
\end{align}
where we have expressed the result in terms of the imaginary error
function $\erfi$.
\end{corollary}
\begin{proof}
The method of an integrating factor may be used to solve Equation~(\ref{ODE}).  
\end{proof}

\section{Asymptotic growth and distributions}

The asymptotics of the exponential generating functions $E(y,w,1)$ and
$E(y,1,z)$, given in Equations~(\ref{byvonly}) and (\ref{byhonly}),
respectively, can be analyzed using the usual machinery of analytic
combinatorics. Let ${\cal V}_k$ be the random variable defined as the
number of vertical dominoes in a random configuration on the $2\times
k$ grid. Similarly, let ${\cal H}_k$ be the analogous random variable
counting horizontal dominoes. The probability distribution functions
$V_{k,v}\coloneqq P({\cal V}_k=v)$ and $H_{k,h}\coloneqq P({\cal
  H}_k=h)$ are computed as follows:
\begin{equation}\nonumber
  V_{k,v} = \frac{1}{(2k-1)!!}\sum_{h\geq 0} D_{k,v,h},\qquad
  H_{k,h} = \frac{1}{(2k-1)!!}\sum_{v\geq 0} D_{k,v,h}.
\end{equation}
Taking derivatives of $E(y,w,1)$ by $w$, we can compute
the factorial moments of $V_{k,v}$. We note that
\begin{equation}\nonumber
  \left[y^k\right] \frac{\partial^m E(y,w,1)}{\partial w^m}\biggr|_{w=1} =
  \left[y^k\right]\frac{y^m}{\sqrt{1-2y}}
  \sim \left(\frac{1}{2}\right)^m \frac{(2k-1)!!}{k!},
\end{equation}
where $\sim$ indicates asymptotic growth in $k$, and so the
$m^{\text{th}}$ factorial moment of $V_{k,v}$ is asymptotically
$(1/2)^m$, consistent with a Poisson distribution of mean $1/2$. A
similar argument can be made for $H_{k,h}$, where the corresponding
mean is found to be $1$. Indeed, Kreweras and Poupard \cite{KP} (see
also Cameron and Killpatrick \cite{CK}) proved that the asymptotic
factorial moments for the distribution of dominoes on the $1\times 2k$
grid graph are all equal to one; the case of horizontal dominoes on
the $2\times k$ grid graph must have the same asymptotic behavior,
since the matchings differ only at a single site: the vertices shown
in red in Figure~\ref{unfold}.

We expect, in the $k\to\infty$ limit, the occurrences of vertical and
horizontal dominoes to be independent, and so the distribution
$P_{k,p}\coloneqq P({\cal H}_k+{\cal V}_k=p)$ of dominoes (vertical or
horizontal) should also be Poisson with mean $1/2 + 1 = 3/2$. We
calculate $P_{k,p}$ as follows:
\begin{equation}\nonumber
P_{k,p} = \frac{1}{(2k-1)!!}\sum_{\substack{v,h\geq 0\\v+h=p}} D_{k,v,h}.
\end{equation}
The present author \cite[Section 2, Theorem 1]{DY} proved that the mean of this
distribution, exact in $k$, is given by
\begin{equation}\nonumber
  \sum_{p\geq 0} p\, P_{k,p} = \frac{3k-2}{2k-1},
\end{equation}
which gives the expected result of $3/2$ in the $k\to \infty$
limit. We do not have an expression for the exponential generating
function for the $D_{k,v,h}$, and so cannot follow the same line of
reasoning used above to establish the equality of the remaining
factorial moments. A different strategy\footnote{The author thanks
  Stephan Wagner for providing this proof.} may be used,
however, to prove that, indeed,
\begin{theorem}\label{threehalves}
\begin{equation}\nonumber
  \lim_{k\to\infty}
  P_{k,p}
  \simeq \frac{e^{-3/2}}{p!}
  \left(\frac{3}{2}\right)^p.
\end{equation}
\end{theorem}
\begin{proof}
  We begin by noting that $P_{k,p}$ is (up to the denominator $(2k-1)!!$) the
  coefficient of $y^kz^p$ in $D(y,z,z)$, i.e.
  \begin{align}\nonumber
    P_{k,p} = &[y^kz^p] 
    \sum_{j=0}^\infty
    \frac{(2j-1)!!}{(2k-1)!!}
  \frac{y^j\,(1-(1-z)y)^j}
       {(1+(1-z)y)^j\left(1+(1-z)y-(1-z)^2y^2\right)^{j+1}}\\\nonumber
       =&\sum_{j=0}^k
    \frac{(2j-1)!!}{(2k-1)!!}[y^{k-j}z^p]
  \frac{(1-(1-z)y)^j}
       {(1+(1-z)y)^j\left(1+(1-z)y-(1-z)^2y^2\right)^{j+1}}.
    \end{align}
  We now define the coefficients $a_{j,n}$ as follows:
  \begin{equation}\nonumber
    a_{j,n}  \coloneqq [x^n]\frac{(1-x)^j}{(1+x)^j(1+x-x^2)^{j+1}}.
    \end{equation}
  We then have that
  \begin{equation}\nonumber
    P_{k,p} = \sum_{j=0}^k  \frac{(2j-1)!!}{(2k-1)!!}[y^{k-j}z^p]
    \sum_{n\geq 0} a_{j,n} (1-z)^n y^n
    = \sum_{j=0}^k  \frac{(2j-1)!!}{(2k-1)!!}[z^p]
     a_{j,k-j} (1-z)^{k-j} .
  \end{equation}
  Let $P_k(z)$ be defined as follows:
  \begin{equation}\nonumber
    P_k(z) \coloneqq \sum_{p\geq 0} P_{k,p} z^p = \sum_{j=0}^k
    \frac{(2j-1)!!}{(2k-1)!!}a_{j,k-j} (1-z)^{k-j} =
   \sum_{\ell=0}^k
    \frac{(2k-2\ell-1)!!}{(2k-1)!!}a_{k-\ell,\ell} (1-z)^{\ell}.
  \end{equation}
  
  We now proceed to prove
  \begin{equation}\nonumber
    \lim_{k\to \infty} P_k(z) = e^{3(z-1)/2},
  \end{equation}
  which is equivalent to the statement of the theorem. We accomplish
  this by placing bounds on the $a_{k-\ell,\ell}$. Note that
  \begin{align}\nonumber
    a_{k-\ell,\ell} &= [x^\ell] \frac{1}{1+x-x^2}\left(
    \frac{1-x}{(1+x)(1+x-x^2)}\right)^{k-\ell}\\\nonumber
&    =(-1)^\ell [x^\ell]
    \frac{1}{1-x-x^2}\left( \frac{1+x}{(1-x)(1-x-x^2)}\right)^{k-\ell}.
  \end{align}
  The following inequalities hold coefficient-by-coefficient:
  \begin{equation}\nonumber
    1 \leq \frac{1}{1-x-x^2} \leq \frac{1}{1-3x},
\quad\text{and}\quad
    1+3x \leq \frac{1+x}{(1-x)(1-x-x^2)} \leq \frac{1}{1-3x}.
  \end{equation}
  It then follows that
  \begin{equation}\nonumber
    [x^\ell](1+3x)^{k-\ell} \leq (-1)^\ell a_{k-\ell,\ell} \leq [x^\ell] (1-3x)^{-(k-\ell+1)},
  \end{equation}
  or, equivalently
  \begin{equation}\nonumber
    3^\ell {k-\ell \choose \ell}  \leq (-1)^\ell a_{k-\ell,\ell} \leq 3^\ell {k \choose \ell}  .
  \end{equation}
 In the $k\to\infty$ limit these bounds become equal. In order to
 establish the form of $P_k(z)$ in this limit, we consider
 \begin{equation}\nonumber
     \frac{(2k-2\ell-1)!!}{(2k-1)!!} 3^\ell {k-\ell \choose \ell} \leq
     \frac{(2k-2\ell-1)!!}{(2k-1)!!} (-1)^\ell a_{k-\ell,\ell} \leq
     \frac{(2k-2\ell-1)!!}{(2k-1)!!} 3^\ell {k \choose \ell} ,
  \end{equation}
 and note that the bounds in this inequality are equal to $(3/2)^\ell/\ell!$
 in the limit, thus
 \begin{equation}\nonumber
   \lim_{k\to\infty} \frac{(2k-2\ell-1)!!}{(2k-1)!!} (-1)^\ell a_{k-\ell,\ell} =
   \frac{3^\ell}{2^\ell \ell!}.
 \end{equation}
It remains to show that (the limit of) the sum in the definition of
$P_k(z)$ is convergent. For this we note that 
\begin{align}\nonumber
  \left| \frac{(2k-2\ell-1)!!}{(2k-1)!!} (-1)^\ell a_{k-\ell,\ell} (z-1)^\ell \right|
  &\leq \frac{(2k-2\ell-1)!!}{(2k-1)!!} 3^\ell {k \choose \ell}
  \left|z-1\right|^\ell\\\nonumber
  &=\frac{3^\ell|z-1|^\ell}{\ell!}\prod_{j=0}^{\ell-1}
  \frac{k-j}{2k-2j-1}\leq\frac{3^\ell|z-1|^\ell}{\ell!},
\end{align}
and, further,
\begin{equation}\nonumber
 \sum_{\ell\geq 0}  \frac{3^\ell|z-1|^\ell}{\ell!} = e^{3|z-1|}
\end{equation}
is a convergent series; by dominated convergence it follows that 
\begin{align}\nonumber
\lim_{k\to\infty} P_k(z) = \sum_{\ell\geq 0} \lim_{k\to\infty}
\frac{(2k-2\ell-1)!!}{(2k-1)!!} (-1)^\ell a_{k-\ell,\ell} (1-z)^\ell = \sum_{\ell\geq 0}
\frac{3^\ell}{2^\ell \ell!} (z-1)^\ell = e^{3(z-1)/2}.
  \end{align}

\end{proof}

\section{Acknowledgements}
The author would like to thank the anonymous referee for a very
careful reading of the manuscript and for suggesting several
improvements. He also thanks Stephan Wagner for providing the proof to
Theorem~\ref{threehalves}.

\bigskip
\hrule
\bigskip

\noindent {\it 2010 Mathematics Subject Classification:} 
Primary 05A15; 
Secondary 05C70, 60C05. 

\noindent \emph{Keywords:} 
Fibonacci number, Fibonacci tree, Domino
  tiling, Perfect matching, Chord diagram, Brauer diagram
\bigskip
\hrule
\bigskip

\noindent (Concerned with sequences \seqnum{A000045},
\seqnum{A046741}, \seqnum{A055140}, \seqnum{A079267},
\seqnum{A178523}, \seqnum{A265167}, \seqnum{A318243},
\seqnum{A318244}, \seqnum{A318267}, \seqnum{A318268},
\seqnum{A318269}, \seqnum{A318270}, \seqnum{A325753},
and \seqnum{A325754})

\bigskip
\hrule
\bigskip

\vspace*{+.1in}
\noindent

\bigskip
\hrule
\bigskip


\end{document}